\documentclass{amsart}
\usepackage{amsthm}
\usepackage{graphicx}
\usepackage{enumerate}
\usepackage{amsaddr}
\usepackage{amssymb}
\makeatletter                            
\newtheorem*{rep@theorem}{\rep@title}    
\newcommand{\newreptheorem}[2]{%
\newenvironment{rep#1}[1]{%
 \def\rep@title{#2 \ref{##1}}%
 \begin{rep@theorem}}%
 {\end{rep@theorem}}}
\makeatother

\newtheorem{thm}{Theorem}[section]
\newreptheorem{theorem}{Theorem}
\newtheorem{cor}[thm]{Corollary}
\newtheorem{fact}[thm]{Fact}
\newtheorem{theorem}[thm]{Theorem}
\newtheorem{lem}[thm]{Lemma}

\newtheorem{conjecture}[thm]{Conjecture}
\newtheorem{obs}[thm]{Observation}
\theoremstyle{definition}
\newtheorem{defn}[thm]{Definition}
\theoremstyle{remark}

\newtheorem{claim}{Claim}
\numberwithin{equation}{section}

\newcommand{\set}[1]{\left\{#1\right\}}

\newcommand{\M}{\mathcal{M}}
\newcommand{\N}{\mathcal{N}}
\newcommand{\F}{\mathcal{F}}
\newcommand{\A}{\mathcal{A}}
\newcommand{\s}{\mathcal{S}}
\newcommand{\Text}[1]{\text{\textnormal{#1}}}
\newcommand{\spn}{\Text{sp}}

\begin{document}
\setlength{\parskip}{0in}
\title[]{Rainbow sets in the intersection of two matroids}%
\author{Ron Aharoni}%
\address{Department of Mathematics, Technion, Haifa 32000, Israel}%
\email{ra@tx.technion.ac.il}%
\author{Daniel Kotlar}%
\address{Computer Science Department, Tel-Hai College, Upper Galilee 12210, Israel}%
\email{dannykot@telhai.ac.il}%
\author{Ran Ziv}%
\address{Computer Science Department, Tel-Hai College, Upper Galilee 12210, Israel}%
\email{ranziv@telhai.ac.il}%
\setlength{\parskip}{0.075in}

\begin{abstract}

Given sets $F_1, \ldots ,F_n$, a {\em partial rainbow function} is a partial choice function of the sets $F_i$.
A {\em partial rainbow set} is the range of a partial rainbow function. Aharoni and Berger \cite{AhBer} conjectured that if $\M$ and $\N$ are matroids on the same ground set, and $F_1, \ldots ,F_n$ are pairwise disjoint sets of size $n$ belonging to $\M \cap \N$, then there exists a rainbow set of size $n-1$ belonging to $\M \cap \N$. Following an idea of Woolbright and Brower-de Vries-Wieringa, we prove that there exists such a rainbow set of size at least $n-\sqrt{n}$.

\end{abstract}
\maketitle
\section{Introduction}\label{sec1}

As in the abstract, a {\em partial rainbow function} of a family of sets $\F=(F_1,\ldots,F_n)$ is a partial choice function. A {\em partial rainbow set} is the range of a rainbow function, so it is a set consisting of at most one element from each $F_i$, where repeated elements are considered distinct (so, in this terminology a rainbow set is in fact a multiset). A full rainbow set, in which elements are chosen from all $F_i$,  is  called plainly a {\em rainbow set}. Strengthening a conjecture of Brualdi and Stein \cite{BruRys,Stein75}, Aharoni and Berger \cite{AhBer} made the following conjecture:

\begin{conjecture} \label{ab} $n$  matchings of size $n+1$ in a bipartite graph have a rainbow matching (namely, a rainbow set that is a matching).
\end{conjecture}

This conjecture easily implies:
\begin{conjecture} \label{abm} $n$  matchings of size $n$ in a bipartite graph have a partial rainbow matching of size $n-1$.
\end{conjecture}

The Brualdi-Stein conjecture is that every Latin square of order $n$  possesses a partial transversal of size $n-1$, namely $n-1$ entries lying in different rows and columns, and containing different symbols. (This is a natural variation on a conjecture of Ryser \cite{Ryser67}, that an odd Latin square has a full transversal). Each of the $n$ rows of a Latin square can be considered in a natural way as a matching of size $n$ between columns and symbols, and applying Conjecture \ref{abm} to these matchings yields the Brualdi-Stein conjecture.

Lower bounds of order $n-o(n)$ were proved in both problems. Hatami and Shor \cite{HatShor08} proved that in a Latin square of order $n$ there exists a partial transversal of size at least $n-11.053\log^2n$.
 Woolbright  \cite{Wool78} and independently Brouwer, de Vries and Wieringa \cite{brouwer78} proved (in effect) that $n$ matchings in a bipartite graph have a partial rainbow matching of size at least $n-\sqrt{n}$.

Aharoni and Berger \cite{AhBer} extended Conjecture \ref{abm} to matroids, as follows:

\begin{conjecture} \label{abmatroids} Let $\M$ and $\N$ be two matroids on the same vertex set. Any $n$ pairwise disjoint sets of size $n$, belonging to $\M \cap \N$, have a partial rainbow set of size $n-1$ belonging to $\M \cap \N$.
\end{conjecture}

Conjecture \ref{abm} is the special case where both $\M$ and $\N$ are partition matroids. (Here the term \emph{partition matroid} refers to a direct sum of uniform matroids, each of rank 1.) The aim of this paper is to prove the parallel of the  Woolbright-Brower-de Vries-Wieringa result for Conjecture \ref{abm}. We shall prove:

\begin{theorem}\label{main}
Any $n$ pairwise disjoint sets of size $n$ belonging to $\M \cap \N$ have a partial rainbow set of size at least $n-\sqrt{n}$ belonging to $\M \cap \N$.
\end{theorem}

\section{Matroid preliminaries}
Throughout the paper we shall use the notation  $A+x$ for $A\cup\set{x}$ and $A-x$ for $A\setminus\set{x}$.

Recall that
a collection $\M$ of subsets of a set $S$ is a \emph{matroid} if it is hereditary and it satisfies an augmentation property: If $A, B\in\M$ and $|B|>|A|$, then there exists $x\in B\setminus A$ such that $A+x\in \M$. Sets in $\M$ are called {\em independent} and sets not belonging to $\M$ are called {\em dependent}.
  A maximal independent set is called a \emph{basis}.
An element $x\in \s$ is \emph{spanned} by $A$ if either $x\in A$ or $I+x\not\in \M$ for some independent set $I\subseteq A$.
The set of elements that are spanned by $A$ is denoted by $\spn(A)$, or $\spn_\M(A)$ if the identity of the matroid $\M$ is not clear from the context.
A \emph{circuit} is a minimal dependent set. We shall use some basic facts on matroids, that can be found, for example, in the books of  Oxley~\cite{Oxley11} and Welsh~\cite{Welsh76}.

\begin{fact}\label{fact1}
If $I$ is independent and $I+x$ is dependent, then there exists a unique minimal subset $C_\M(I,x)$ of $I$ that spans $x$. \end{fact}
We shall call $C_\M(I,x)$ the \emph{$\M$-support} of $x$ in $I$.

\begin{fact}\label{fact2}
Let $A\in\M$, $x\in\spn(A)$, and $a\in C_\M(A,x)$. Then $A+x-a\in\M$ and $\spn(A+x-a)=\spn(A)$.
\end{fact}

\begin{fact}\label{fact3}
If $C_1$ and $C_2$ are circuits with $e\in C_1\cap C_2$ and $f\in C_1\setminus C_2$ then there exists a circuit $C_3$ such that $f\in C_3\subseteq (C_1\cup C_2)-e$.
\end{fact}

The following is an immediate corollary of the augmentation property:
\begin{fact}\label{fact4}
Let  $I,J$ be independent sets in $\M$. If $|I|<|J|$, then there exists $J_1\subseteq J\setminus I$ such that $I\cup J_1\in\M$ and $|I\cup J_1|=|J|$.
\end{fact}

\begin{defn}\label{def1}
Let $\M$ and $\N$ be two matroids on the same ground set $S$. We call a set $F\subseteq S$ an \emph{independent matching} if $F\in \M\cap \N$.
A rainbow set for a family $\F=\{F_1, F_2, \ldots,F_n\}$  of independent matchings is called a {\em  rainbow independent matching} if it belongs to $\M \cap \N$.
\end{defn}

\section{Proof of Theorem \ref{main}}\label{sec2}

Let $\F=(F_1, \ldots ,F_n)$ be a family of disjoint sets belonging to $\M \cap \N$. Let $R$ be a partial rainbow matching for $\F$ of maximal size. Let $t=|R|$ and $\delta=n-t$. Without loss of generality we may assume that $|R\cap F_i|=1$ for $i=1,\ldots,t$.

We shall construct a sequence of sets $(A_1,\ldots,A_\delta)$ such that for all $i=1,\ldots,\delta$ the following holds:
\begin{equation}\label{eq3:1}
A_i\subseteq F_{t+i},
\end{equation}
\begin{equation}\label{eq3:2}
A_i\subseteq \spn_\M(R),
\end{equation}
and
\begin{equation}\label{eq3:3}
|A_i|\geq i\delta.
\end{equation}

Suppose that we  succeed in constructing such a sequence.
By \eqref{eq3:1}  $A_\delta\in\M$ and by \eqref{eq3:2}  $A_\delta\subseteq \spn_\M(R)$. By \eqref{eq3:3}, applied to $i=\delta$, we therefore have $t=|R|\ge|A_\delta|\ge \delta^2$. Clearly, we may assume that $t<n$. Since $\delta=n-t$, it  follows that $t>n-\sqrt{n}$, as stated in the theorem.

{\em Construction of the sets $A_i$}.
We construct the sets $A_i$ inductively, associating with them sets $R_i$, so that  $R_1,\ldots,R_\delta$ are disjoint,  $R_i\subseteq R$ and $|R_i|\ge \delta$ for all $i=1,\ldots,\delta$.
Since $|F_{t+1}|=n$ and $|R|=t$, there exists, by Fact~\ref{fact4}, a set $A_1\subseteq F_{t+1}\setminus R$ such that $|A_1|=\delta$ and $R\cup A_1\in\N$.
By the maximality property of $R$ we have  $A_1\subseteq\spn_\M(R)$. Since $|A_1|=\delta$ and $|R|=t$, there exists, again by Fact~\ref{fact4}, a subset $R'\subset R$ of size $t-\delta$ such that $A_1\cup R'\in \M$ and $|A_1\cup R'|=t$. Let $R_1=R\setminus R'$. We have $R\setminus R_1\cup A_1\in \M$ and $|R_1|=\delta$.

For the inductive step, assume that   $R_1,R_2,\ldots,R_j$ are pairwise disjoint subsets of $R$, each of size at least $\delta$, and  $A_1,A_2,\ldots,A_j$ satisfy the conditions (\ref{eq3:1}), (\ref{eq3:2}) and (\ref{eq3:3}), for $i=1,\ldots,j$.
Denote $R^{k}=R\setminus\cup_{i=1}^{k-1}R_i$  for $k=2,\ldots$.
Notice that $|R^{j+1}|\leq t-j\delta$. Since $F_{t+j+1}\in\N$ and $|F_{t+j+1}|=n$ it follows from Fact~\ref{fact4} that there exists $A_{j+1}\subseteq F_{t+j+1}$ such that $R^{j+1}\cup A_{j+1}\in\N$ and $|R^{j+1}\cup A_{j+1}|=n$.  We have $|A_{j+1}|=n-|R^{j+1}|\geq n-(t-j\delta)=(j+1)\delta$. We see that $A_{j+1}$ satisfies (\ref{eq3:1}) and (\ref{eq3:3}). The following lemma implies that (\ref{eq3:2}) also holds for $A_{j+1}$.

\begin{lem}\label{lem2}
If $j<\delta$ then $A_{j+1}\subseteq\spn_\M(R)$.
\end{lem}
Before proving Lemma~\ref{lem2} let us indicate how it is used to complete the inductive construction of $R_{j+1}$. We use the following observation:

\begin{obs}\label{obs:0}
Let $I$ be an independent set of size $t$ in a matroid $\M$ and suppose $J\subseteq\spn(I)$ has size $n>t$. If $K\subset J$ satisfies $J\setminus K\in \M$, then $|K|\ge n-t$.
\end{obs}

Assuming Lemma~\ref{lem2}, we have (*) $A_{j+1}\subseteq\spn_\M(R)$. We also have $|R^{j+1}\cup A_{j+1}|=n=|R|+\delta$. Hence $|R^{j+1}|\ge\delta$ (If $|R^{j+1}|<\delta$ then $|A_{j+1}|>|R|$, contradicting (*)). Let $R_{j+1}\subset R^{j+1}$ be of minimal size such that $R^{j+1}\setminus R_{j+1}\cup A_{j+1}\in\M$. By Observation~\ref{obs:0} we have $|R_{j+1}|\geq\delta$, as required.

The proof of Lemma~\ref{lem2} is done by an alternating path argument.
\begin{defn}\label{def1}
A \emph{colorful alternating path} (CAP) of length $k$, relative to $R$, consists of
\begin{enumerate}
  \item [(i)] A set $\{b_0,b_1,\ldots,b_k\}$ of distinct elements of the ground set $S$, where each $b_i$ belongs to some $A_j\in\A$ and distinct $b_i$'s belong to distinct $A_j$'s.
  \item [(ii)] A set of distinct elements $\{r_1,\ldots,r_k\}\subseteq R$ such that
  \begin{enumerate}
    \item [(P$_\M$)] $R-r_1+b_1-r_2+b_2-\cdots-r_k+b_k\in\M$ and $\spn_\M(R-r_1+b_1-r_2+b_2-\cdots-r_k+b_k)=\spn_\M(R)$.
    \item [(P$_\N$)] $R+b_0-r_1+b_1-r_2-\cdots+b_{k-1}-r_k\in\N$ and $\spn_\N(R+b_0-r_1+b_1-r_2+\cdots+b_{k-1}-r_k)=\spn_\N(R)$.
  \end{enumerate}
\end{enumerate}

If, in addition, $R+b_0-r_1+b_1-r_2+b_2-\cdots-r_k+b_k\in\M\cap \N$ then the CAP is called \emph{augmenting}.
\end{defn}

Since the $b_i$'s belong to distinct $F_{t+j}$'s we have:

\begin{obs}\label{obs:1}
If $R$ is of maximal size then no augmenting CAP relative to $R$ exists.
\end{obs}

In order to extend our alternating path we shall need the following lemma:
\begin{lem}\label{lem3}
Let $\M$ be a matroid. Let $I\in\M$ and $X=\set{x_1,\ldots,x_k}\subseteq I$ and $Y=\set{y_1,\ldots,y_k}\subseteq \spn_\M(I)\setminus I$ be such that $\spn_\M\left(\left(I\setminus X\right) \cup Y\right)=\spn_\M(I)$.
Suppose $y_{k+1}\in \spn_\M(I)\setminus I$ and $x_{k+1}$ are such that $x_{k+1}\in C_\M(I,y_{k+1})\setminus X$ and $x_{k+1}\not\in C_\M(I,y_i)$ for all $i=1,\ldots,k$. Then $x_{k+1}\in C_\M(\left(I\setminus X\right) \cup Y,y_{k+1})$.
\end{lem}
\begin{proof}[Proof of Lemma~\ref{lem3}]
Suppose, for contradiction, that $x_{k+1}\not\in C_\M(\left(I\setminus X\right) \cup Y,y_{k+1})$.
Let $C'= C_\M(I,y_{k+1})+y_{k+1}$ and $C''= C_\M(\left(I\setminus X\right) \cup Y,y_{k+1})+y_{k+1}$. Then, by Fact~\ref{fact3}, there exits a circuit $C\subseteq C'\cup C''$, such that $x_{k+1}\in C$ and $y_{k+1}\not\in C$. Choose such a circuit $C$ with $|C\cap Y|$ minimal.
Since $I$ is independent $C$ must contain at least one element $y_j\in Y\cap C''$. Using Fact~\ref{fact3} again, since $x_{k+1}\not\in C_\M(I,y_j)$, there exists a circuit $\tilde{C}\subseteq C\cup (C_\M(I,y_j)+y_j)$ such that $x_{k+1}\in \tilde{C}$ and $y_{j}\not\in \tilde{C}$. We have $|\tilde{C}\cap Y|<|C\cap Y|$, contradicting the minimality property of $C$.
\end{proof}

\begin{proof}[Proof of Lemma~\ref{lem2}]
We shall show how the existence of some $i$, $1\leq i \leq\delta$, such that $A_i\not\subseteq\spn_\M(R)$ yields an augmenting CAP relative to $R$. This will contradict the maximality of $R$, by Observation~\ref{obs:1}.

Let $\{A_i\}$, $\{R_i\}$ and $\{R^i\}$ be defined as above. Recall that for all $i=1,\ldots,\delta$,
\begin{equation}\label{eq:p1}
   R^i=R\setminus\cup_{j=1}^{i-1}R_j,
\end{equation}
\begin{equation}\label{eq:p2}
    A_i\subseteq F_{t+i}\Text{ satisfies }R^{i}\cup A_{i}\in\N\Text{ and }|R^{i}\cup A_{i}|=n
\end{equation}
and
\begin{equation}\label{eq:p3}
    R_i\subseteq R^i\Text{ is of minimal size such that }R^i\setminus R_i\cup A_i\in \M.
\end{equation}

Assume, for contradiction, that $m$, $1\leq m\leq\delta$, is the minimal index such that $A_m\not\subseteq\spn_\M(R)$ and let $a\in A_m$ be such that $R+a\in\M$.
We shall construct a CAP, relative to $R$, starting from $a$. Let $b_0=a$. We have
\begin{equation}\label{eq1}
    R+b_0\in\M
\end{equation}
and, since no augmenting CAP relative to $R$ exists, we must have $b_0\in\spn_\N(R)$. Let $j$ be the maximal index such that $b_0\in\spn_\N(R^j)$. Since $b_0\in A_m$ and, by (\ref{eq:p2}), $R^m\cup A_m\in\N$, we obtain $b_0\not\in\spn_\N(R^m)$. Thus, $j<m$. Since $R_j=R^j\setminus R^{j+1}$, it follows from the maximality of $j$ that $C_\N(R^j,b_0)\cap R_j\neq\emptyset$.
By Fact~\ref{fact2}, there exists $r_1\in R_j$ such that $R+b_0-r_1\in\N$ and
\begin{equation}\label{eq1:5}
    \spn_\N(R+b_0-r_1)=\spn_\N(R).
\end{equation}

Since $j<m$, we have, by the minimality of $m$, that $A_j\subseteq \spn_\M(R)$. By the minimality of $R_j$ (see (\ref{eq:p3})) there exists $x\in A_j$ such that $r_1\in C_\M(R,x)$ (otherwise $A_j\cup R^{j+1}+r_1\in \M$). Let $l\leq j$ be minimal such that $A_l$ contains an element $b_1$ satisfying $r_1\in C_\M(R,b_1)$.
By Fact~\ref{fact2}, we have $R-r_1+b_1\in\M$ and $\spn_\M(R-r_1+b_1)=\spn_\M(R)$. This, combined with (\ref{eq1}), implies that $R+b_0-r_1+b_1\in\M$. Thus, a CAP of length 1 was created.

Now, suppose that we managed to construct a CAP of length $k$. We shall show that if the CAP is not augmenting, then it can be extended. Denote $R_\M(k)=R-r_1+b_1-r_2+b_2-\cdots-r_k+b_k$ and $R_\N(k)=R+b_0-r_1+b_1-r_2+\cdots+b_{k-1}-r_k$. Note that
\begin{equation}\label{eq3}
    R_\M(k)+b_0= R_\N(k)+b_k.
\end{equation}

\begin{claim}\label{claim:1}
$b_k\in\spn_\N(R)$.
\end{claim}

\begin{proof}[Proof of Claim~\ref{claim:1}]
\renewcommand{\qedsymbol}{}
By (P$_\M$), we have $\spn_\M(R_\M(k))=\spn_\M(R)$. Hence, from (\ref{eq1}) we have $R_\M(k)+b_0\in\M$. Also, by (P$_\N$), we have $\spn_\N(R_\N(k))=\spn_\N(R)$. Assume, for contradiction, that $R+b_k\in\N$. Then, $R_\N(k)+b_k\in\N$, and by (\ref{eq3}) we obtain an augmenting CAP, contradicting the maximality property of $R$.
\end{proof}

Assuming Claim 1, let $p$ be the maximal index such that $b_k\in \spn_\N(R^p)$. By (\ref{eq:p1}), $p$ is the minimal index such that $C_\N(R,b_k)\cap R_p\ne \emptyset$.
Let $r_{k+1}\in C_\N(R,b_k)\cap R_p$. By Fact~\ref{fact2}, $R+b_k-r_{k+1}\in\N$ and $\spn_\N(R+b_k-r_{k+1})=\spn_\N(R)$.

\begin{claim}\label{claim:2}
Let $q$ be the index such that $b_k\in A_q$. Then, $p<q$.
\end{claim}

\begin{proof}[Proof of Claim~\ref{claim:2}]
\renewcommand{\qedsymbol}{}
By (\ref{eq:p2}), $R^q\cup A_q\in \N$ and thus, $b_k\not\in \spn_\N(R^q)$. Since $b_k\in \spn_\N(R^p)$ we conclude that $R^q\subsetneq R^p$, which implies that $p<q$.
\end{proof}

\begin{claim}\label{claim:3}
There exists $x\in A_p$ such that $r_{k+1}\in C_\M(R,x)$.
\end{claim}

\begin{proof}[Proof of Claim~\ref{claim:3}]
\renewcommand{\qedsymbol}{}
Assume the opposite. Then $A_p\cup R^{p+1}+r_{k+1}\in \M$. This contradicts the minimality property of $R_p$ (see (\ref{eq:p3})).
\end{proof}

Let $l$ be minimal such that $A_l$ contains an element $b_{k+1}$ satisfying $r_{k+1}\in C_\M(R,b_{k+1})$. By Claim 3, $l\le p$. This, together with Claim 2, yields

\begin{equation}\label{eq2:1}
  \text{if $b_i\in A_u$ and $b_j\in A_v$ with $i<j$, then $v<u$,}
\end{equation}
and
\begin{equation}\label{eq2:2}
  \text{if $r_i\in R_u$ and $r_j\in R_v$ with $i<j$, then $v<u$.}
\end{equation}

\begin{claim}\label{claim:4}
$r_{k+1}\not\in C_\N(R,b_i)$  for all  $i=0,\ldots,k-1$.
\end{claim}

\begin{proof}[Proof of Claim~\ref{claim:4}]
\renewcommand{\qedsymbol}{}
Let $j\in\{1,\ldots,k\}$. In the construction described above, the element $r_j$ was chosen from $R_u$, where $u$ is minimal such that $C_\N(R, b_{j-1})\cap R_u\ne \emptyset$. Recall that $r_{k+1}\in R_p$. Thus, by (\ref{eq2:2}), we have $p< u$, and hence $C_\N(R, b_{j-1})\cap R_p = \emptyset$, which implies the claim.
\end{proof}

By applying Lemma~\ref{lem3} to the sequences $\set{r_1,\ldots,r_k, r_{k+1}}$ and $\set{b_0,\ldots,b_{k-1}, b_k}$, it follows that $r_{k+1}\in C_\N(R_\N(k),b_k)$.
By Fact~\ref{fact2}, it follows that
\begin{equation}\label{eq5}
\begin{split}
  & R_\N(k)+b_k-r_{k+1}\in \N\Text{,  and } \\
  & \spn_\N(R_\N(k)+b_k-r_{k+1})=\spn_\N(R_\N(k))=\spn_\N(R).
\end{split}
\end{equation}

\begin{claim}\label{claim:5}
$r_{k+1}\in C_\M(R_\M(k),b_{k+1})$.
\end{claim}

\begin{proof}[Proof of Claim~\ref{claim:5}]
\renewcommand{\qedsymbol}{}
Let $i\in\{1,\ldots,k\}$. In the construction described above, the element $b_i$ was chosen from $R_u$, where $u$ is minimal such that $A_u$ contains an element $b_i$ such that $r_i\in C_\M(R,b_i)$. Recall that $b_{k+1}$ was chosen from $A_l$, and by (\ref{eq2:1}), $l<u$. Thus, $r_i\not\in C_\M(R,b_{k+1})$. Since this is true for any $i\in\{1,\ldots,k\}$, we have $C_\M(R,b_{k+1})\cap\set{r_1,\ldots,r_k}=\emptyset$, and hence, $C_\M(R_\M(k),b_{k+1})=C_\M(R,b_{k+1})$. Since $b_{k+1}$ was chosen so that $r_{k+1}\in C_\M(R,b_{k+1})$, the claim follows.
\end{proof}

Assuming Claim 5, by Fact~\ref{fact2}, we have
\begin{equation}\label{eq6}
\begin{split}
  & R_\M(k)+b_{k+1}-r_{k+1}\in\M\Text{,  and } \\
  & \spn_\M(R_\M(k)+b_{k+1}-r_{k+1})=\spn_\M(R_\M(k))=\spn_\M(R).
\end{split}
\end{equation}

By (P$_\M$), (P$_\N$), (\ref{eq5}) and (\ref{eq6}), the CAP was extended to the length of $k+1$.

By (\ref{eq2:1}) and (\ref{eq2:2}), the process must end, yielding an augmenting CAP. This completes the proof of Lemma~\ref{lem2} and hence of Theorem~\ref{main}.

\end{proof}
\section{Independent partial transversals in Matroidal Latin Squares}\label{sec3}
Let $\M$ be matroid of rank $n$ defined on a ground set $S$. A \emph{Matroidal Latin Square (MLS)} of degree $n$ over $\M$ was defined in \cite{KZ12} as an $n\times n$ matrix whose entries are from $S$, such that each row and column is a basis of $\M$. (After publication, the authors found out that a similar object had been introduced earlier by Chappell \cite{Chappell99}.)  Note that the notion of MLS generalizes the notion of Latin square, as a Latin square is an MLS over a partition matroid (that is, a direct sum of uniform matroids, each of rank 1). Analogously to Norton's definition for row Latin square in \cite{Norton52}, we define a \emph{row MLS}, as an $n\times n$ matrix whose entries are from $S$, such that each row is a basis of $\M$. Thus, every MLS is a row MLS.

An \emph{independent partial transversal} in an MLS, or in a row MLS, $A$, is an independent subset of entries of $A$ where no two of them lie in the same row or column of $A$. It was conjectured in \cite{KZ12} that every MLS of degree $n$ has an independent partial transversal of size $n-1$. It was shown there that, in general, we cannot expect to find a partial independent transversal of size $n$. The lower bound set in \cite{KZ12} for the size of a partial independent transversal in an MLS was $\lceil2n/3\rceil$. Theorem~\ref{main} yields a significant improvement for that bound:

\begin{cor}\label{cor1}
Every row MLS of degree $n$ has an independent partial transversal of size at least  $n-\sqrt{n}$.
\end{cor}
\begin{proof}
Let $A$ be a row MLS of degree $n$ over a matroid $\M$. The result follows from Theorem~\ref{main} by taking $\N$ as the partition matroid defined by the columns of $A$.
\end{proof}

\section*{Acknowledgments}
The authors thank two anonymous referees for their insightful comments and for their substantial contribution to the clarity of the manuscript.

\bibliographystyle{amsplain}
\providecommand{\bysame}{\leavevmode\hbox to3em{\hrulefill}\thinspace}
\providecommand{\MR}{\relax\ifhmode\unskip\space\fi MR }
\providecommand{\MRhref}[2]{%
  \href{http://www.ams.org/mathscinet-getitem?mr=#1}{#2}
}
\providecommand{\href}[2]{#2}

\end{document}